\documentclass[12pt]{amsart}
\numberwithin{equation}{section}
\usepackage{amsmath,amsthm,amsfonts,amscd,eucal}

\usepackage{xypic}
\usepackage{caption}
\usepackage{subfig}
\usepackage{graphicx}

\usepackage{amssymb}
\def\cf{{\mathcal F}}
\def\cg{{\mathcal G}}
\def\ch{{\mathcal H}}

\def\cam{{\mathcal M}}

\def\cs{{\mathcal S}}

\def\ga{{\mathfrak A}}

\def\gam{{\mathfrak M}}

\def\bc{{\mathbb C}}

\def\bn{{\mathbb N}}

\def\br{{\mathbb R}}

\def\bz{{\mathbb Z}}

\def\a{\alpha}
\def\b{\beta}
\def\g{\gamma}  
\def\d{\delta}

\def\l{\lambda} 

\def\m{\mu}

\def\r{\rho}
\def\s{\sigma} 

\def\f{\varphi}  
  
\def\om{\omega} \def\Om{\Omega}

\newtheorem{theorem}{Theorem}[section]
\newtheorem{lemma}[theorem]{Lemma}

\newtheorem{proposition}[theorem]{Proposition}

\def\sp{\mathop{\rm sp}}

\def\supp{\mathop{\rm supp}}

\def\di{{\rm d}}

\begin{document}

\title[]
{Vacuum distribution, norm and spectral properties for sums of monotone position operators}
\author{Vitonofrio Crismale}
\address{Vitonofrio Crismale\\
Dipartimento di Matematica\\
Universit\`{a} degli studi di Bari\\
Via E. Orabona, 4, 70125 Bari, Italy}
\email{\texttt{vitonofrio.crismale@uniba.it}}
\author{Yun Gang Lu}
\address{Yun Gang Lu\\
Dipartimento di Matematica\\
Universit\`{a} degli studi di Bari\\
Via E. Orabona, 4, 70125 Bari, Italy}
\email{\texttt{yungang.lu@uniba.it}}
\date{\today}

\begin{abstract} We investigate the spectrum for partial sums of $m$ position (or gaussian) operators on monotone Fock space based on $\ell^2(\mathbb{N})$. In the basic case of the first consecutive operators, we prove it coincides with the support of the vacuum distribution. Thus, the right endpoint of the support gives their norm. In the general case, we get the last property for norm still holds. As the single position operator has the vacuum symmetric Bernoulli law, and the whole of them is a monotone independent family of random variables, the vacuum distribution for partial sums of $n$ operators can be seen as the monotone binomial with $n$ trials. It is a discrete measure supported on a finite set, and we exhibit recurrence formulas to compute its atoms and probability function as well. Moreover, lower and upper bounds for the right endpoints of the supports are given.

\vskip0.1cm\noindent
 \\

\bigskip

\end{abstract}

\maketitle

\section{introduction}

Position operators on Fock spaces are the self-adjoint part of creators or annihilators with the same test function. In non commutative probability they are also called generalised gaussian operators, and in the monotone case \cite{DeGLu,Lu,Mur2} are the most natural examples of monotone independent random variables \cite{Mur}. As a consequence, their partial sums, up to usual rescaling, weakly converge in the vacuum state to the standard (i.e. centered with unit variance) arcsine law, namely the probability distribution with density $\nu(\di x)=\frac{1}{\pi\sqrt{2-x^2}}\di x$ on $(-\sqrt{2},\sqrt{2})$.

Many results have been obtained in the last years in the monotone kingdom, such as monotone convolution and monotone central limit theorems \cite{Mu, Mur, CFL2}, monotone cumulants and monotone infinite divisibility \cite{HasSa, Has}, and the list above is far to be complete. Monotone Fock spaces, as prominent examples of interacting Fock spaces, were first investigated in \cite{Lu}, whereas in \cite{Boz} the author highlighted the relations between monotone creation and annihilation operators and Pusz-Woronowicz twisted operators \cite{PW}. More recently, the study of distributional symmetries on monotone stochastic processes built on the concrete $C^*$-algebra of creation and annihilation operators on monotone Fock space was started in \cite{CFL, CFG}. The basic idea of monotone Fock spaces is a suitable deformation of the usual $n$th scalar product on the $n$th particle space of the full Fock space. Namely, the new scalar product is induced by the orthogonal projection onto the linear space spanned by some increasingly ordered (w.r.t. a linear order on the index set) elements of the canonical basis of the full Fock space. As a special case of the so-called Yang-Baxter-Hecke quantisation \cite{Boz}, monotone creation and annihilation operators sometimes exhibit common features with the $q$-deformed case, with $-1<q<1$ (see, e.g. \cite{BKS}). As an example, the reader is referred to \cite{CrLu}. The situation radically changes for monotone stochastic processes invariant under some distributional symmetries, which behave in a completely different way \cite{CFL,CFG}. Furthermore, in the $q$-deformed case, the vacuum vector is separating for the von Neumann algebra generated by all of the gaussian operators, whereas in the monotone case it was proved in \cite{CFL} that the commutant for the same algebra is trivial. As a consequence, even for a single position operator, one cannot directly deduce that the support of the moments distribution in the vacuum state covers the whole spectrum, the latter condition being equivalent to the faithfulness of the vector state. Up to our knowledge, the spectral properties for sums of monotone position operators have not yet been investigated. Here we present a path to achieve information on the spectrum.

Namely, after denoting $s_i:=a(e_i)+a^\dag(e_i)$, $i\in\bn$ the gaussian operator on the monotone Fock space built on $\ell^2(\mathbb{N})$, we prove that the Radon measure induced by the vacuum vector on the spectrum of the unital commutative $C^*$-algebra generated by $S_m:=\sum_{i=1}^m s_i$, is \emph{basic} \cite{Dix} for any $m$. This property in particular entails the above measure is supported on the whole Gelfand spectrum. As the latter results to be homeomorphic to the spectrum $\s(S_m)$ of $S_m$, it turns out $\s(S_m)$ is covered by the support of the vacuum law. Consequently, one figures out that the norm of $S_m$ is exactly the right endpoint of the support. Since arbitrary sums of $m$ position operators are identically distributed in the vacuum state, one naturally wonders if even they share their norm with $S_m$. Although it is not immediate, we give an affirmative answer.

The above arguments therefore lead us to investigate firstly the vacuum distribution of $S_m$. Recall that any $s_i$ is endowed with the symmetric Bernoulli law in the vacuum and, as previously mentioned, the collection of such operators is a family of monotone independent random variables. This suggests that the measure of any partial sum can be viewed as the \emph{monotone binomial} distribution. Here, using the monotone convolution, we highlight that any law is a symmetric measure supported on a finite subset of the reals, and give recurrence formulas for computing weights and atoms.

The paper is organised as follows. After some preliminary results comprised in Section \ref{sec1bis}, monotone binomial laws represent the main argument of Section \ref{sec2}. Using some results of \cite{Has}, in Proposition \ref{law1} we show recurrence formulas for atoms and probability functions, and the section ends with an estimation of the right endpoints, say $r_n$, of the supports. Although affected by a small error, this directly gives the size of the support, avoiding the longer recurrence formula, and as a biproduct, it allows us to achieve the sequence $\big(\frac{r_n}{\sqrt{n}}\big)_n$ converges (increasingly from the left) to $\sqrt{2}$, a consistent result with the monotone Central Limit Theorem \cite{Mur}. Finally, in Section \ref{sec4} we state the main theorem, i.e. the Radon measure defined by the vacuum vector is basic on the spectrum of the unital $C^*$-algebra $\cs_m$ generated by any $S_m$. As previously noticed, this entails that the spectrum of $S_m$ and the support of the vacuum distribution coincide, and in particular provides the value of the norm. The proof is obtained after showing the vacuum is a cyclic vector for the commutant of $\cs_m$ for any $m$. This crucial property is not generally verified when one handles with a general sum involving $m$ monotone position operators, as shown in the paper. Nevertheless, a suitable operator direct sum decomposition gives the norm in this case, which turns out to be equal to that of $S_m$, as one naturally foresees. The paper ends with an appendix where we briefly show how the method of moments generating function induces a nice relation among the atoms of the vacuum law of $S_m$ and those of the distributions of $S_1,\ldots, S_{m-1}$. The result, presented in Proposition \ref{lem2}, refines an existing one in \cite{Has} based  on monotone convolution, and is added here for the convenience of the reader.

\section{preliminaries}
\label{sec1bis}
In this section we mainly recall some definitions and features which will used throughout the paper.

\subsection{basic measures on Gelfand spectrum}

Let $\ch$ be a separable Hilbert space with inner product $\langle \cdot,\cdot \rangle$, and $\ga$ an abelian $C^*$-algebra of operators on $\ch$. Recall the spectrum $\sp(\ga)$ of $\ga$ is the $*$-weakly locally compact (compact if $\ga$ is unital) space of characters. If $C_0(\sp(\ga))$ is the collection of continuous complex-valued functions on $\sp(\ga)$ vanishing at infinity and $f\in C_0(\sp(\ga))$, we get $T_f$ as the unique element in $\ga$ realising the \emph{Gelfand isomorphism}, \emph{i.e.} the normed $*$-algebra isomorphism between $C_0(\sp(\ga))$ and $\ga$ such that $\f(T_f)=f(\f)$ (see, \emph{e.g.} \cite{Dix} for details). For any $f\in C_0(\sp(\ga))$ and $x,y\in \ch$, we denote by $\nu_{x,y}$ the spectral measure such that $\nu_{x,y}(f):=\langle T_f x, y\rangle$. Notice that when $\ga$ is unital, one replaces in the above lines $C_0(\sp(\ga))$ with the algebra $C(\sp(\ga))$ of continuous functions on $\sp(\ga)$.

\noindent A (Radon) positive measure $\m$ on $\sp(\ga)$ is called \emph{basic} \cite{Dix} if for any subset in $\sp(\ga)$ to be locally $\m$-negligible, it is necessary and sufficient to be locally $\nu_{x,x}$-negligible for all $x\in \ch$.

\noindent If $\m$ is a basic measure, any other basic measure on $\sp(\ga)$ is measure equivalent to $\m$ (\emph{i.e.} they are mutually absolutely continuous). Moreover, since the union of the supports of the $\nu_{x,x}$ is dense in $\sp(\ga)$ and any $\nu_{x,x}$ is absolutely continuous w.r.t $\m$, in particular one has that $\m$ is supported on the whole spectrum of $\ga$.

\subsection{monotone independence}
Let $\m$ be a probability measure defined on the Borel $\s$-field over $\br$. The moment sequence associated with $\mu$ is denoted by $(m_n(\mu))_{n\geq1}$. Recall that for each $z\in\bc$
$$
\cam_\m(z):=\sum_{n=0}^{\infty}z^nm_n(\mu)
$$
is called moment generating function, which is considered as a formal power series if the series is not absolutely convergent.

From now on $\mathbb{C}^+$ and $\mathbb{C}^-$ will be the the upper and lower complex half-planes, respectively.  The Cauchy transform of $\m$ is defined as
\begin{equation*}
\cg_{\mu}(z):=\int_{-\infty}^{+\infty}\frac{\mu(dx)}{z-x},
\end{equation*}
i.e.
$$
\cg_{\mu}(z)=\frac{1}{z}\cam_\m\bigg(\frac{1}{z}\bigg).
$$
The map
$$
H_{\mu}(z):=\frac{1}{\cg_{\mu}(z)}
$$
is called the reciprocal Cauchy transform of $\m$.
$\cg_{\mu}(z)$ is analytic in $\bc\setminus \supp({\mu})$, and since $\cg_{\mu}(\overline{z})=\overline{\cg_{\mu}(z)}$, we can restrict its domain on $\mathbb{C}^+\cup \mathbb{R}$, where it uniquely determines $\m$.

The reciprocal Cauchy transform $H_{\mu}(z)$ plays an important role when one has to compute the distribution of a sum of monotone independent random variables \cite{Mur}, as we will see below.

Recall that an algebraic probability space is a pair $(\ga,\f)$, where $\ga$ is a unital $*$-algebra and $\f$ a state on $\ga$, i.e. a linear functional defined on $\ga$ such that $\f(a^*a)\geq 0$ for any $a\in\ga$, and $\f(1_{\ga})=1$. In this case any $a\in \ga$ is called a random variable. Consider a linearly ordered family $(\ga_i)_{i\in I}$ of $*$-subalgebras of $\ga$, where the index set $I$ is linearly ordered by the relation $<$. The family $(\ga_i)_{i\in I}$ is said to be \emph{monotone independent} if
$$
\f(a_1\cdots a_i \cdots a_n)=\f(a_i)\f(a_1\cdots a_{i-1}a_{i+1}\cdots a_n),
$$
when $a_{i-1}<a_i$ and $a_{i+1}<a_i$, with the elimination of one of the inequalities when $i=1$ or $i=n$. A family of random variables
is said to be monotonically independent if the family of subalgebras generated by each random variable is monotone independent. We first recall the following
\begin{theorem}[\cite{Mu}, Theorem 3.1]
\label{conv}
Let $a_{1},a_{2},\ldots,a_{n}\in \ga$ be monotonically independent self-adjoint
random variables, in the natural order, over a $*$-algebraic probability space $(\ga,\f)$. If $\m_{a_i}$ is the probability distribution of ${a_i}$  under the state $\f$, then
\begin{equation}
\label{htran}
H_{\m_{a_1+a_2+\ldots + a_n}}(z)=H_{\m_{a_1}}(H_{\m_{a_2}}(\cdots H_{\m_{a_n}}(z)\cdots)).
\end{equation}
\end{theorem}
\noindent Moreover, Theorem 3.5 in \cite{Mu} ensures that for any pair of probability measures $\m,\nu$ on $\mathbb{R}$, there exists a unique
distribution $\r$ on $\mathbb{R}$ such that
$$
H_{\r}(z) = H_\m(H_\nu(z)).
$$
$\r$ is called the monotonic convolution of $\m$ and $\nu$, and denoted by $\m\rhd \nu$. Monotone convolution is associative and affine in the first argument, and Theorem \ref{conv} entails the law for any partial sum of monotone independent random variables is the monotone convolution of the marginal distributions.

\section{the vacuum law for sums of position operators}
\label{sec2}
The section is devoted to present the vacuum distribution for partial sums of position operators in discrete monotone Fock space. It is well known that position operators are a family of monotone independent random variables \cite{Mur}. As a consequence, it appears quite natural to perform our investigation using the monotone convolution \cite{Mu}. Furthermore, since any gaussian operator is symmetrically Bernoulli distributed, Theorem 3.1 and Corollary 3.3 in \cite{Has} give that the vacuum distribution for the sum of $m$ position operators is a discrete measure with exactly $2^m$ atoms, and a formula for computing the weights. The main result of the section is Proposition \ref{law1}, where we collect the above results, and give a recurrence formula for computing the atoms of the law.

We first recall some useful features on discrete monotone Fock space, the reader being referred to \cite{CFL,CFL2, Lu, Mur} for further details.

For $k\geq 1$, denote $I_k:=\{(i_1,i_2,\ldots,i_k) \mid i_1< i_2 < \cdots <i_k, i_j\in \mathbb{N}\}$. The discrete monotone Fock space is the Hilbert space $\cf_m:=\bigoplus_{k=0}^{\infty} \ch_k$, where for any $k\geq 1$, $\ch_k:=\ell^2(I_k)$, and $\ch_0=\mathbb{C}\Om$, $\Om$ being the Fock vacuum. Borrowing the terminology from the physical language, we call each
$\ch_k$ the $k$th-particle space and denote by $\cf^o_m$ the dense linear manifold of finite particle vectors in $\cf_m$, that is
$$
\cf^o_m:=\bigg\{\sum_{n=0}^{\infty} c_n\xi_n \mid \xi_n\in \ch_n,\,\, c_n\in \mathbb{C}\,\,\, \text{s.t.}\,\, c_n=0\,\,\, \text{but a finite set}  \bigg\} \,.
$$
Let $(i_1,i_2,\ldots,i_k)$ be an increasing sequence of natural integers. The generic element of the canonical basis of $\cf_m$ is denoted by $e_{(i_1,i_2,\ldots,i_k)}$. Very often, we write $e_{(i)}$ as $e_i$ to simplify the notations.
The monotone creation and annihilation operators are respectively given, for any $i\in \mathbb{N}$, by $a^\dag_i\Om=e_i$, $a_i\Om=0$ and
\begin{equation*}
a^\dagger_i e_{(i_1,i_2,\ldots,i_k)}:=\left\{
\begin{array}{ll}
e_{(i,i_1,i_2,\ldots,i_k)} & \text{if}\, i< i_1 \\
0 & \text{otherwise}, \\
\end{array}
\right.
\end{equation*}
\begin{equation*}
a_ie_{(i_1,i_2,\ldots,i_k)}:=\left\{
\begin{array}{ll}
e_{(i_2,\ldots,i_k)} & \text{if}\, k\geq 1\,\,\,\,\,\, \text{and}\,\,\,\,\,\, i=i_1\\
0 & \text{otherwise}. \\
\end{array}
\right.
\end{equation*}
One can check that both $a^\dagger_i$ and $a_i$ have unital norm (see \cite{Boz}, Proposition 8), they are mutually adjoint, and satisfy the following relations
\begin{equation}
\label{comrul}
\begin{array}{ll}
  a^\dagger_ia^\dagger_j=a_ja_i=0 & \text{if}\,\, i\geq j\,, \\
  a_ia^\dagger_j=0 & \text{if}\,\, i\neq j\,.
\end{array}
\end{equation}
In addition, for any $i$ the following identity holds
\begin{equation}
\label{comrul2}
a_ia^\dag_i+\sum_{k\leq i} a^\dag_ka_k=I.
\end{equation}
From now on, for a fixed $i\in\mathbb{N}$ we denote by $s_i$ the sum of creation and annihilation operators with the test function $e_i$, namely
$$
s_i:=a_i+a^\dag_i,
$$
which is generally called the position field operator.
Moreover, for any $m\in\mathbb{N}$ one takes
$$
S_m:=\sum_{i=1}^m s_i.
$$
The distribution $\m_m$ of $S_m$ in the vacuum vector state $\om_\Om:=\langle \cdot\Om, \Om\rangle$ can be deduced using some existing results on monotone convolution \cite{Mu}. First, one notices it is a compactly supported measure on the real line, since $S_m$ is a bounded self-adjoint operator. It is then determined by the moments $u_{m,n}:=\om_{\Om}((S_m)^n)$. As for any $i,n\in\bn$ it is easy to see that $s_i^{2n}=s_i^2$ and $s_i^{2n+1}=s_i$, one has $\m_1=\frac{1}{2}(\d_1+\d_{-1})$, i.e. $\m_1$ is the Bernoulli symmetric law. Concerning the case $\m_n$, $n\geq 2$, the following result gives a recursive formula to compute the atoms and weights for the (necessarily discrete) law of any partial sum of monotone gaussian operators in the vacuum state.
\begin{proposition}
\label{law1}
For any $n\in\mathbb{N}$, the vacuum distribution of $S_I$, for $I:=\{i_1<\cdots < i_n\mid i_j\in \bn\}$ is the discrete measure $\displaystyle \m_n:=\sum_{k=1}^{2^n}b_k^{(n)}\d_{r_k^{(n)}}$, where
$r_1^{(1)}=1,\,\,\,\,\, r_2^{(1)}=-1,\,\,\,\,\,  b_1^{(1)}=b_2^{(1)}=1/2$. Furthermore, for any $j=0,\ldots, 2^{n-1}-1$ one finds
\begin{equation}
\label{ato1}
r_{2j+h}^{(n)}=\frac{r_{j+1}^{(n-1)}+(-1)^h\sqrt{\big(r_{j+1}^{(n-1)}\big)^2+4}}{2}\,, \quad h=1,2
\end{equation}
and for any $k=1,\ldots, 2^n$ one achieves
\begin{equation}
\label{wei1}
b_k^{(n)}=\frac{\prod_{h=1}^{2^{n-1}}\big(r_k^{(n)}-r_h^{(n-1)}\big)}{2\prod_{\substack{h=1 \\ h\neq k}}^{2^{n}}\big(r_k^{(n)}-r_h^{(n)}\big)}\,.
\end{equation}
\end{proposition}
\begin{proof}
We first recall that $(s_i)_{i\geq 1}$ is a family of monotone independent and identically distributed self-adjoint random variables in $(\gam_o,\om_{\Om})$, where $\gam_o$ is the $*$-algebra generated by $\{a_i\mid i\in\bn\}$. Therefore, it is enough to prove the statement for $S_n$, exploiting Theorem \ref{conv}.  Indeed, from \eqref{htran}, it follows
$$
H_{\m_n}(z)=H_{\m_{n-1}}(H_{\m_{1}}(z)),
$$
as monotone convolution is associative, and the reciprocal Cauchy transform of $s_i$ is $H_{s_i}(z)=\frac{z^2-1}{z}$ for any $i$. Thus, any zero, say for simplicity $r$, of $H_{\m_n}$ satisfies the following
$$
\frac{r^2-1}{r}=r_j^{(n-1)},
$$
for each $j=1,\ldots, 2^{n-1}$, and one achieves \eqref{ato1}. Finally, after denoting $\m_{2,n}$ the vacuum distribution of $s_2+\cdots + s_n$, \eqref{htran} gives $\m_n=\m_1\rhd \m_{2,n}$, and \eqref{wei1} follows from Theorem 3.1 in \cite{Has}, as monotone convolution is affine in the first argument.
\end{proof}
Since $\m_1$ is a symmetric measure, a standard induction procedure with \eqref{ato1} and \eqref{wei1} give any $\m_n$ is symmetric too. Furthermore, the $2^n$ points of the support of the vacuum distribution for the sum of $n$ position operators come in inverse pairs, when $n\geq 2$. Indeed, fix $r_0$ an element in the support of $\m_n$, for some $n\geq 1$. Then \eqref{ato1} entails that both
$$
r_1:=\frac{r_0+\sqrt{r_0^2+4}}{2}
$$
and
$$
r_2:=\frac{-r_0+\sqrt{(-r_0)^2+4}}{2}
$$
are in the support of $\m_{n+1}$, and trivially $r_1r_2=1$. As a consequence, one achieves all of the atoms of $\m_n$ just computing half of the positive ones by \eqref{ato1}, as soon as $n$ is at least $2$.

In Figure \ref{fig:distributions} we report the plots for the vacuum laws of $S_n$, $n\leq 4$.
\begin{figure}[b!]
\centering
\subfloat[][\emph{Distribution for $n=1$}.]
{\includegraphics[width=.45\textwidth]{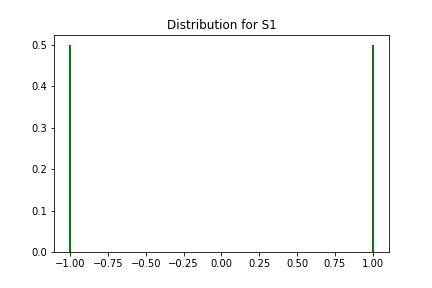}} \quad
\subfloat[][\emph{Distribution for $n=2$}.]
{\includegraphics[width=.45\textwidth]{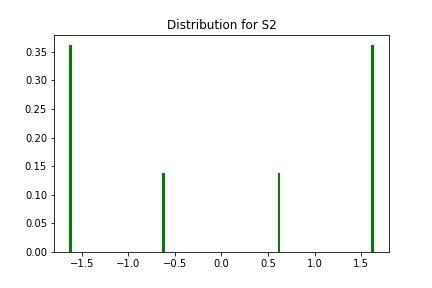}} \\
\subfloat[][\emph{Distribution for $n=3$}.]
{\includegraphics[width=.45\textwidth]{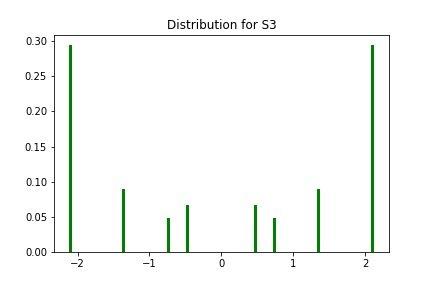}} \quad
\subfloat[][\emph{Distribution for $n=4$}.]
{\includegraphics[width=.45\textwidth]{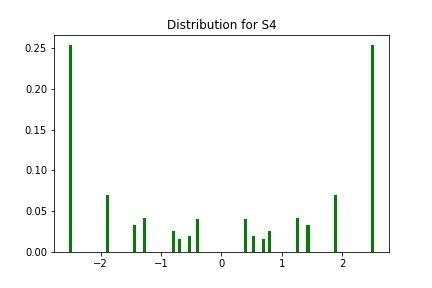}}
\caption{Vacuum distribution of $S_n$, $n=1,\ldots,4$.}
\label{fig:distributions}
\end{figure}

In what follows we give a small error approximation for the extreme values of the support of $\mu_n$, a result appearing useful if compared with the content of Section \ref{sec4}.
\begin{proposition}
Under the notations introduced above, for any $n\geq 1$ one has
\begin{equation}
\label{estim}
\sqrt{2n-\sqrt{2n}}\leq r^{(n)}_{2^n} < \sqrt{2n}
\end{equation}
\end{proposition}
\begin{proof}
We start by showing the right inequality, which is true for $n=1$ as $r_2^{(1)}=1$. Suppose now it holds for any $k\leq n$. Then, since for any $n$ one straightforwardly sees
\begin{equation}
\label{ltrot}
\sqrt{2n}+\sqrt{2n+4}< 2\sqrt{2(n+1)},
\end{equation}
\eqref{ato1} and \eqref{ltrot} give
\begin{align*}
r^{(n+1)}_{2^{n+1}}&=\frac{1}{2}\bigg(r^{(n)}_{2^n}+\sqrt{\big(r^{(n)}_{2^n}\big)^2+4}\bigg) \\
&\leq \frac{1}{2}\bigg(\sqrt{2n}+\sqrt{2n+4}\bigg) < \sqrt{2(n+1)},
\end{align*}
after recalling the map $\br \ni x\mapsto x+\sqrt{x^2+4}$ is increasing.
Proving the inequality $\sqrt{2n-\sqrt{2n}}\leq r^{(n)}_{2^n}$ is longer. It indeed holds for $n=1$. For the remaining cases, after denoting $m:=2n$, we firstly show that
\begin{equation}
\label{rtrot}
\sqrt{m-\sqrt{m}}+\sqrt{\big(m-\sqrt{m}\big)^2+4}\geq 2\sqrt{m+2-\sqrt{m+2}}.
\end{equation}
This is indeed satisfied when $n=1$. Since for any $m\geq 3$, one further has
$$
\sqrt{\big(m-\sqrt{m}\big)^2+4}\geq \sqrt{m-\sqrt{m}+4},
$$
when $n\geq 2$, the inequality \eqref{rtrot} holds true if
$$
\sqrt{m-\sqrt{m}}+\sqrt{\big(m-\sqrt{m}\big)+4}\geq 2\sqrt{m+2-\sqrt{m+2}}.
$$
In fact, after squaring one finds this is equivalent to
$$
\sqrt{\big(m-\sqrt{m}\big)\big(m-\sqrt{m}+4\big)}\geq m+2+\sqrt{m}-2\sqrt{m+2},
$$
and a further squaring gives
$$
\sqrt{m}+\frac{m+3}{m+2}-\frac{m+2+\sqrt{m}}{\sqrt{m+2}}\leq 0.
$$
The last inequality is equivalent to
$$
4m^4+12m^3-4m^2-24m+1\geq 0,
$$
which is automatically satisfied since the map $f(x):=4x^4+12x^3-4x^2-24x+1$ is strictly increasing in $[3,+\infty)$, and $f(3)>0$.

Suppose now that the left inequality in \eqref{estim} holds for each $k\leq n$. One can extend its validity to $k=n+1$ by means of \eqref{rtrot} and \eqref{ato1}.
\end{proof}
As a consequence of \eqref{estim}, the condition that the atoms of $\m_n$ come in inverse pair suggests that the littlest positive of them approaches $0$ for $n$ going to $+\infty$, and moreover
$$
\lim_n \frac{r^{(n)}_{2^n}}{\sqrt{n}}=\sqrt{2}.
$$
The last result agrees with the central limit theorem for monotonically independent random variables \cite{Mur}. Namely, the sum of $n$ position operators, rescaled by a factor $\frac{1}{\sqrt{n}}$, weakly converges to the arcsine law supported in $(-\sqrt{2},\sqrt{2})$ for $n\rightarrow \infty$. In addition, \eqref{estim} suggests $\bigg(\frac{r^{(n)}_{2^n}}{\sqrt{n}}\bigg)_n$ approaches $\sqrt{2}$ from the left, and it is an increasing sequence, since by \eqref{ato1} it is not difficult to prove that for any $n$
$$
\bigg(\frac{2}{\sqrt{n}}-\frac{1}{\sqrt{n+1}}\bigg)r^{(n)}_{2^n} < \sqrt{\frac{(r^{(n)}_{2^n})^2+4}{n+1}}.
$$
The reader is referred to \cite{CGW} for similar results in the so-called weakly monotone case.

\section{the norm for sums of position operators}
\label{sec4}

As previously pointed out, our approach to compute the norm for partial sums of position operators on monotone Fock space provides an investigation of their spectrum. To this goal, we begin with the definition of the right creators and annihilators on $\cf_m$. For any $i\in\mathbb{N}$, take $b^\dag_i\Om=e_i$, $b_i\Om=0$ and
\begin{align*}
&b^{\dagger}_ie_{(i_1,i_2,\ldots,i_k)}:=\left\{
\begin{array}{ll}
e_{(i_1,\ldots,i_k,i)} & \text{if}\, i> i_k\,,\\
0 & \text{otherwise}\,,\\
\end{array}
\right. \\
&b_ie_{(i_1,i_2,\ldots,i_k)}:=\left\{
\begin{array}{ll}
e_{(i_1,i_2,\ldots,i_{k-1})} & \text{if}\, k\geq 1\,\,\,\,\,\, \text{and}\,\,\,\,\,\, i=i_k\,,\\
0 & \text{otherwise}\,.\\
\end{array}
\right.
\end{align*}
Since they are continuous on $\cf_m^o$, they can be uniquely extended to the whole $\cf_m$ where they are mutually adjoint, and endowed with unital norm. For any $i$, the right position operator is defined as $r_i:=b_i+b^\dag_i$.

Right creators and annihilators are a powerful tool to study the von Neumann algebra generated by position operators in the $q$-deformed case, $-1<q<1$ \cite{BKS}. There, the commutant of the von Neumann algebra is generated by right position operators, and the vacuum is a cyclic vector. The latter property entails that the support of the vacuum distribution for sums of position operators covers their spectrum. In the monotone case, the $C^*$-algebra generated by position operators, say $\mathfrak{S}_m$, is irreducible, as follows (up to replacing $\bz$ with $\bn$) from \cite{CFL}, Propositions 5.9 and 5.13. Consequently, $\Om$ is not cyclic for the commutant, and we are forced to reduce to suitable $C^*$-subalgebras of $\mathfrak{S}_m$.

One preliminary notices that the abelian $C^*$-algebra generated by $s_1$ contains the identity operator on $\cf_m$, as $s^2_1=I$. The same happens when one takes the $C^*$-algebras generated by $S_2$ and $S_3$. Indeed, one has $3S_2^2-S_4^4=I$, and $7S_3^2-13S_3^4+7S_3^6-S_3^8=I$, respectively. More in general, using the following identities
\begin{align*}
&s_i^{2n}=s_i^2,\,\,\,\,\,\,\,\,\,\,\,\,\,\,s_i^{2n+1}=s_i,\\
&s_i^2s_{i+1}=s_{i+1},\\
&s_{i+1}^2+\sum_{j=1}^i s_js_{j+1}^2s_j=I,
\end{align*}
coming from \eqref{comrul}, \eqref{comrul2}, and Lemma 5.4 and Proposition 5.13 of \cite{CFL}, we conjecture for any $n$ there exists a finite sequence of scalars $(\a_k)_k$ such that
$$
\sum_{k=1}^{\frac{(n-1)n}{2}+1}\a_{2k}S_n^{2k}=I.
$$
Its proof is not in the aim of these notes and it does not affect the following results. Thus, from now on we denote by $\cs_n$ the abelian $C^*$-algebra on $\cf_m$ generated by $S_n$ for any $n$, and tacitely suppose it contains $I$. Next theorem shows that $\cs_n$ are indeed the $C^*$-subalgebras of $\mathfrak{S}_m$ we are looking for, and is crucial to prove that $\m_n$ is supported on the whole spectrum of $S_n$.
\begin{theorem}
\label{norm}
For any $n\in\mathbb{N}$, the spectral measure $\nu_{\Om,\Om}$ is basic on $\sp(\cs_n)$.
\end{theorem}
\begin{proof}
Since \cite{Dix}, Ch. 7 Proposition 2, it is enough to prove that $\Om$ is a cyclic vector for the commutant $\cs_n'$ of $\cs_n$, $n\in\mathbb{N}$. Namely, by usual approximation arguments, we need to show that for any $k\geq 0$, $i_1<i-2<\cdots i_k$, there exists $T\in \cs_n'$ such that $T\Om=e_{i_1}\otimes \cdots e_{i_k}$, where $k=0$ gives $\Om$.

\noindent To this aim, we firstly check that for any $k\in\mathbb{N}$
$$
[s_k,r_{k+j}]=0,\,\,\,\,\,\,\,\, j=0,1,2,\ldots
$$
Indeed, after fixing $j\geq 0$, one has
\begin{equation*}
s_kr_{k+j}\Om=r_{k+j}s_k\Om=\left\{
\begin{array}{ll}
                              e_k\otimes e_{k+j} & \text{if}\,\,\,\,\, j\geq 1, \\
                              \Om & \text{if}\,\,\,\,\, j=0. \\
                            \end{array}
                            \right.
\end{equation*}
Moreover, for any $e_l\in\ch$
\begin{equation*}
s_kr_{k+j}e_{l}=r_{k+j}s_ke_l=\left\{
\begin{array}{llll}
                                                            e_{k+j} & \text{if}\,\,\,\,\, l=k, \\
                              e_{k}\otimes e_l\otimes e_{k+j} & \text{if}\,\,\,\,\, k<l<k+j, \\
                              e_k & \text{if}\,\,\,\,\, l=k+j, \\
                              0 & \text{otherwise}. \\
                            \end{array}
                            \right.
\end{equation*}
Finally, for $n\geq 2$ and $i_1<\cdots i_n$,
$$
s_kr_{k+j}(e_{i_1}\otimes \cdots \otimes e_{i_n})
$$
and
$$
r_{k+j}s_k(e_{i_1}\otimes \cdots \otimes e_{i_n})
$$
are both equal to
\begin{equation*}
\left\{
\begin{array}{lllll}
                                                            e_{i_2}\otimes \cdots \otimes e_{i_{n-1}} & \text{if}\,\,\,\,\, j\geq1, k+j=i_n, k=i_1, \\
                              e_k\otimes e_{i_1}\otimes \cdots \otimes e_{i_{n-1}} & \text{if}\,\,\,\,\, j\geq1, k+j=i_n, k<i_1, \\
                              e_{i_2}\otimes \cdots \otimes e_{i_{n}}\otimes e_{k+j} & \text{if}\,\,\,\,\, j\geq1, k+j>i_n, k=i_1, \\
                              e_{k}\otimes e_{i_1}\otimes \cdots \otimes e_{i_{n}}\otimes e_{k+j} & \text{if}\,\,\,\,\, j\geq1, k+j>i_n, k<i_1,\\
                              0 & \text{otherwise}. \\
                            \end{array}
                            \right.
\end{equation*}
As a consequence, for each $j\geq 0$
\begin{equation}
\label{commu1}
[S_n,r_{n+j}]=0
\end{equation}
i.e. any $r_{n+j}$ commutes with each element of the $*$-algebra generated by $S_n$. Thus, a standard approximation argument gives $r_{n+j}\in\cs_n'$.
Furthermore, one has
\begin{equation}
\label{commu2}
I\Om=\Om
\end{equation}
and
$$
r_{i_n}\cdots r_{i_2}r_{i_1}\Om=e_{i_1}\otimes e_{i_2}\otimes \cdots \otimes e_{i_n},
$$
for $i_1<i_2<\cdots < i_n$ and any $n\geq 1$. Then it then turns out $\Om$ is cyclic for $\cs_1'$. Since
\begin{equation}
\label{commu3}
(S_2-r_2)\Om=e_1,
\end{equation}
from \eqref{commu1} and \eqref{commu3}, it follows $e_i=T_i\Om$, where $T_i\in\cs_2'$ and $i\in\bn$, whereas \eqref{commu2} gives $\Om=I\Om$. Moreover, for any $i_1<i_2<\cdots < i_n$, it results
$$
r_{i_n}\cdots r_{i_2}\overline{r}_{i_1}\Om=e_{i_1}\otimes e_{i_2}\otimes \cdots \otimes e_{i_n},
$$
where $\overline{r}_{i_1}=(S_2-r_2)$ or $\overline{r}_{i_1}=r_{i_1}$, according to whether $i_1=1$ or $i_1>1$. As a consequence, $\overline{\cs_2'\Om}=\cf_m$.

The general case $n\geq 3$ can be performed as follows. As a first step one shows that each $e_j$, $j=1,\ldots, n-1$ is obtained by the action onto the vacuum of suitable operators belonging to $\cs_n'$. Pursuing this requires the replacement of \eqref{commu3} with the more general
\begin{equation}
\label{commua}
A_n\Om:=(S_n-r_n)\Om=\sum_{i=1}^{n-1}e_i,
\end{equation}
and the detection of some operators in $\cs_n'$ mapping the vacuum into $e_j$, $j=2,\ldots n-1$. In the latter case $e_1$ is obtained from \eqref{commua}.

\noindent More in detail, for $n=3$, a possible choice for the above mentioned operators goes through
$$
B_3:=S_3(S_3-r_3)-2I.
$$
$B_3$ indeed gives $e_2$ by means of $S_3B_3\Om$. This, together with \eqref{commua}, \eqref{commu1} and \eqref{commu2} allows us to find $\Om$  and the elements of the canonical basis of $\ell^2(\bn)$. Finally, it results to be cyclic for $\cs_3'$, since for $n\geq 2$
\begin{equation}
\label{base}
r_{i_n}\cdots r_{i_3}\overline{r}_{i_2}\overline{r}_{i_1}\Om=e_{i_1}\otimes e_{i_2}\otimes \cdots \otimes e_{i_n},
\end{equation}
where
$$
\overline{r}_{i_2}\overline{r}_{i_1}:=\left\{\begin{array}{ll}
                                               B_3& \text{if}\,\,\, i_1=1,i_2>2, \\
                                               r_{i_2}(A_3-S_3B_3) & \text{if}\,\,\, i_1=1,i_2>2, \\
                                               r_{i_2}S_3B_3 & \text{if}\,\,\, i_1=2, \\
                                               r_{i_2}r_{i_2} & \text{if}\,\,\, i_1>2.
                                             \end{array}
                                             \right.
$$
The case $n\geq 4$ appears more complicated. To achieve $e_j$, $j=1,\ldots,n-1$ it seems natural to start with the natural generalisation of $B_3$, given by
$$
B_n:=S_n(S_n-r_n)-(n-1)I.
$$
Then we observe that
\begin{align}
\label{commub}
\begin{split}
S_nB_n\Om&=S_n\bigg(\sum_{1\leq i<j\leq n-1} e_i\otimes e_j\bigg) \\
&=\sum_{i=2}^{n-1}(i-1)e_i + \sum_{1\leq i<j<k\leq n-1} e_i\otimes e_j\otimes e_k.
\end{split}
\end{align}
By means of \eqref{commua} and \eqref{commub}, one has
\begin{equation*}
\big[(n-2)A_n-S_nB_n\big]\Om=\sum_{i=1}^{n-2}(n-i-1)e_i\,\,\,- \sum_{1\leq i<j<k\leq n-1} e_i\otimes e_j\otimes e_k.
\end{equation*}
After defining
\begin{align*}
C_n\Om:=&\bigg(S_n[(n-2)A_n-S_nB_n] - \frac{(n-2)(n-3)}{2}r_n^2\bigg)\Om \\
=& \sum_{i=1}^{n-3}\sum_{k=i+1}^{n-2}(n-k-i)e_i\otimes e_k-\sum_{i=2}^{n-2}\sum_{k=i+1}^{n-1}(i-1)e_i\otimes e_k,
\end{align*}
one notices it is useful to erase the term $e_{n-2}\otimes e_{n-1}$ from the r.h.s. above to reach our goal. This is achieved by taking
$$
D_n\Om:=[(n-3)B_n+C_n]\Om
$$
As a consequence,
\begin{equation*}
S_nD_n\Om=\sum_{i=2}^{n-1}\g_ie_i + \sum_{i=1}^{n-4}\sum_{j=i+1}^{n-3}\sum_{k=j+1}^{n-1}\l_{ijk}e_i\otimes e_j\otimes e_k,
\end{equation*}
for suitable integers $\g_i$ and $\l_{ijk}$. Next step consists in erasing $e_{n-1}$ from the first sum above, a result obtained by computing
$$
(\g_{n-1}A_n-S_nD_n)\Om.
$$
Then one applies again $S_n$ to the quantity above and iterates the procedure. Finally, one recovers the analogue of the r.h.s. of \eqref{commua}, i.e.
$$
\sum_{i=1}^{n-2}\a_ie_i,
$$
for some integers $\a_i$. Using similar arguments as above, one erases $e_{n-2}$, thus reducing the matter to a linear combination of $e_1,\ldots, e_{n-3}$. Several iterations of the same procedure lead us to remove, in the following order, $e_{n-3},\ldots, e_3$, and thus to find a suitable $E_n\in\cs_n'$ such that
$$
E_n\Om:=\b_1e_1+\b_2e_2,
$$
for $\b_1,\b_2\in\mathbb{Z}$. As a consequence,
$$
\big(S_nE_n-(\b_1+\b_2)r_n^2\big)\Om=\b_2e_1\otimes e_2,
$$
and the last equality allows us to get $e_2$, since
\begin{equation}
\label{star}
\b_2e_2=S_n\big(S_nE_n-(\b_1+\b_2)r_n^2\big)\Om.
\end{equation}
The remaining $e_i$, for $i=3,\ldots, n-1$ can be similarly obtained.

The second step consists in finding the remaining elements of the canonical basis of $\cf_m^o$. This is obtained, \emph{mutatis mutandis}, as in \eqref{base}.
\end{proof}
To get a flavour of the above exposed procedure, in the case $n=4$ one finds \eqref{star} has the following form
$$
-e_2=S_4^2\big(S_4-r_4-S_4\big[S_4(S_4-r_4)-3I+S_4(B_4-3I)\big]\big)\Om.
$$
As a consequence, we have the following
\begin{theorem}
\label{norm2}
For any $n\in\mathbb{N}$, one has
$$
\s(S_n)=\supp(\m_n)
$$
and
$$
\|S_n\|=r_{2^n}^{(n)},
$$
where, as usual, $\s$ denotes the spectrum of an operator.
\end{theorem}
\begin{proof}
We first recall that for any $n$, the map $\Theta_n:\sp(\cs_n)\rightarrow \s(S_n)$ s.t. $\Theta_n(\f)=\f(S_n)$ is a homeomorphism. As a consequence, for any $f\in C(\sp(\cs_n))$
\begin{align*}
\int_{\sp(\cs_n)}f(\f)\di\nu_{\Om,\Om}(\f)&=\langle(S_n)_f\Om,\Om\rangle \\
&=\int_{\s(S_n)}(f\circ \Theta_n^{-1})(\f(S_n))\di\m_n(\f(S_n)),
\end{align*}
where $f\mapsto (S_n)_f$ is the Gelfand isomorphism. Therefore,
$\m_n(B)=\nu_{\Om,\Om}(\Theta_n^{-1}(B))$ for any borelian $B$ in $\s(S_n)$. Hence, by Theorem \ref{norm}, $\m_n$ results to be supported on the whole compact $\s(S_n)$. As $S_n=S_n^*$, the last part of the statement follows from Proposition \ref{law1}.
\end{proof}
The property that $\Om$ is cyclic for any $\cs_n'$ is crucial in the proof of Theorem \ref{norm}. Then, one naturally wonders if the vacuum monotone vector is cyclic for the commutant of any $C^*$-algebra generated by a finite sum of gaussian monotone operators. The following example shows it is not generally true.

Let us take the operator $S_{1,3}:=s_1+s_3$ and denote by $\cs_{1,3}$ the unital $C^*$-algebra generated by it. Here we show there does not exist any $T\in (\cs_{1,3})'$ such that $T\Om=e_2$. To this aim, we preliminary notice that for any $\xi:=e_{i_1}\otimes \ldots \otimes e_{i_m}$ of the canonical basis of $\cf_m^o$ such that $i_1\geq 4$, one finds
$$
S_{1,3}\xi=e_1\otimes \xi + e_3\otimes \xi.
$$
As a consequence, we reduce our matter to the action of $S_{1,3}$ on the set
\begin{equation}
\label{base4}
\{\Om,e_1,e_2,e_3, e_1\otimes e_2, e_1\otimes e_3, e_2\otimes e_3, e_1\otimes e_2\otimes e_3\},
\end{equation}
which is represented by the hermitian matrix $A:=(a_{ij})_{i,j=1,\ldots, 8}$ assuming the form
$$
A=\left(
     \begin{array}{cccccccc}
       0 & 1 & 0 & 1 & 0 & 0 & 0 & 0 \\
       1 & 0 & 0 & 0 & 0 & 0 & 0 & 0 \\
       0 & 0 & 0 & 0 & 1 & 0 & 0 & 0 \\
       1 & 0 & 0 & 0 & 0 & 1 & 0 & 0 \\
       0 & 0 & 1 & 0 & 0 & 0 & 0 & 0 \\
       0 & 0 & 0 & 1 & 0 & 0 & 0 & 0 \\
       0 & 0 & 0 & 0 & 0 & 0 & 0 & 1 \\
       0 & 0 & 0 & 0 & 0 & 0 & 1 & 0 \\
     \end{array}
   \right)
$$
Let $T$ be an element in $\cs_{1,3}'$, with $B:=(b_{ij})_{i,j=1,\ldots, 8}$ its representing matrix on the vectors \eqref{base4}.
The condition $[A,B]=0$ immediately gives $b_{1j}\neq 0$ only when $j=1,2,4,6$. Moreover, the same condition implies $b_{31}=b_{32}=b_{34}=b_{36}=0$. Thus $e_2$ does not belong to $\overline{\cs_{1,3}'\Om}$.

Therefore, it turns out that our approach does not give information on the relation between vacuum law and spectrum for general partial sums of position operators. Nevertheless, in the next lines we will show that, as in the case of vacuum distribution, also the norm depends only on the number of operators in the partial sum.

To this aim, fix an integer $m$ and consider the set $I$ given by a sequence of strictly increasing indices $1\leq i_1<i_2<\cdots < i_n=m$ such that there exists $i_j$, $j=1,\ldots, n$, for which $i_{j}\neq i_{j-1}+1$, where $i_0:= 0$. As usual, we denote $S_I:=\sum_{h=1}^n s_{i_h}$ and we look for the norm of $S_I$. If $J:=\{1, \ldots, m\}\setminus I$, one has
\begin{equation}
\label{direct}
\cf_m=\cf_J\oplus \cf_J^\perp
\end{equation}
where $\cf_J$ denotes the closure in $\cf_m$ of the subspace $\cf_m^o$ generated by the vectors
$$
\{e_{j_1}\otimes \cdots \otimes e_{j_k} \mid k\geq 1, j_1<\cdots <j_k\,\,\,\, \text{and}\,\,\,\, j_l\in J\,\,\, \text{for some}\,\,\, l\},
$$
and $\cf_J^\perp$ is the orthogonal complement of $\cf_J$. The dense subspace of $\cf_J^\perp$ built similarly as $\cf_m^o$ will be denoted by $(\cf_J^\perp)^o$.
After noticing that $\mathbb{C}\Om$ belongs to $\cf_J^\perp$, it not difficult to check that $S_I$ leaves invariant both the subspaces $\cf_J$ and $\cf_J^\perp$. From now on, we denote by $S_{I,J}^\perp$ and $S_{I,J}$ the restrictions of $S_I$ on $\cf_J^\perp$ and $\cf_J$, respectively.
\begin{proposition}
\label{norm3}
Under the above notations, one has $\s(S_{I,J}^\perp)=\s(S_n)$.
\end{proposition}
\begin{proof}
 Since the vacuum distributions of $S_{I,J}^\perp$ and $S_n$ are equal, as in Theorem \ref{norm} it is enough to prove that $\Om$ is cyclic for the commutant of the unital $C^*$-algebra $\cs_{I,J}^\perp$ generated by $S_{I,J}^\perp$. In fact, in this case the thesis will follow arguing as in the proof of Theorem \ref{norm2}. To this purpose, we first recall that the right hand side partial shift based on $h$ is the one-to-one map $\theta_h:\mathbb{N}\rightarrow \mathbb{N}$ such that
$$
\theta_h(k):=\left\{\begin{array}{ll}
                      k & \text{if}\,\, k<h \\
                      k+1 & \text{if}\,\, k\geq h.
                    \end{array}
                    \right.
$$
If $j_1<j_2<\ldots <j_{m-n}$ are the elements of $J$, we denote by $\theta_{J}$ the composition $\theta_{j_{m-n}}\circ\cdots \circ \theta_{j_{1}}$. Consider $U_{\theta, J}\in\mathcal{B}(\cf_m, \cf_J^\perp)$
 such that
\begin{align*}
& U_{\theta, J}\Om:=\Om \\
&U_{\theta, J}e_B:=e_{\theta_{J}(B)}\,,
\end{align*}
$e_B$ being a generic element of the canonical basis of $\cf_m$. As
\begin{align*}
&U_{\theta, J}^*\Om=\Om \\
&U_{\theta, J}^*e_C= e_{\theta_{J}^{-1}(C)},
\end{align*}
where $e_C$ is an arbitrary element of the canonical basis of $\cf_J^\perp$, one achieves $U_{\theta, J}$ is unitary, i.e.
\begin{align}
\begin{split}
\label{iso}
&U_{\theta, J}^*U_{\theta, J}=I \\
&U_{\theta, J}U_{\theta, J}^*=I_{\cf_J^\perp}.
\end{split}
\end{align}
Denote $\cs_{I,J}^\perp$ the $C^*$-algebra generated by $S_{I,J}^\perp$ and $I_{\cf_J^\perp}$. After recalling that $S_{I,J}^\perp$ leaves invariant $\cf_J^\perp$, one finds for any $k$
\begin{equation*}
S_n^k=U_{\theta, J}^*(S_{I,J}^\perp)^k U_{\theta, J}.
\end{equation*}
This gives
\begin{equation}
\label{commu4}
[(S_{I,J}^\perp)^k, U_{\theta, J}TU_{\theta, J}^*]=0
\end{equation}
for any $T\in\cs_n'$.
Then it follows $\Om$ is cyclic for $(\cs_{I,J}^\perp)'$, since from \eqref{iso} and \eqref{commu4}
$$
U_{\theta, J}\cs_n'U_{\theta, J}^*\Om=U_{\theta, J}\cf_m^o=(\cf_J^\perp)^o.
$$
\end{proof}
The next result allows us to compute the norm of any sum of monotone gaussian operators.
\begin{proposition}
Under the notation introduced above, one has
$$
\|S_I\|=r^{(n)}_{2^n}
$$
\end{proposition}
\begin{proof}
Indeed, as $S_I$ leaves both the subspaces $\cf_J$ and $\cf_J^\perp$ invariant, from \eqref{direct} one finds
$$
S_I=S_{I,J} \oplus S_{I,J}^\perp.
$$
Fix a generic element $\xi$ in  $\cf_J^o$, i.e.
$$
\xi:=\sum_{h=1}^p\a_h e_{k_1^{(h)}}\otimes \cdots \otimes e_{k_r^{(h)}}\,,
$$
where $p\in \mathbb{N}$, $k_1^{(h)}<\cdots < k_r^{(h)}$ and $k_l^{(h)}\in J$ for some $l$. After recalling that $S_{I,J}$ acts only on $e_{k_1^{(h)}}$, one finds
\begin{align*}
S_{I,J}\xi=& \sum_{1\leq h \leq p,\,\, k_1^{(h)}=i_1}\a_h e_{k_2^{(h)}}\otimes \cdots \otimes e_{k_r^{(h)}} \\
&+\sum_{1\leq h \leq p,\,\, i_1<k_1^{(h)}<i_2}\a_h e_{i_1}\otimes e_{k_1^{(h)}}\otimes \cdots \otimes e_{k_r^{(h)}} \\
&+\sum_{1\leq h \leq p,\,\, k_1^{(h)}=i_2}\a_h(e_{i_1}\otimes e_{k_1^{(h)}}\otimes \cdots \otimes e_{k_r^{(h)}}+ e_{k_2^{(h)}}\otimes \cdots \otimes e_{k_r^{(h)}}) \\
&+ \sum_{1\leq h \leq p,\,\, i_2<k_1^{(h)}<i_3}\a_h(e_{i_1}+e_{i_2})\otimes e_{k_1^{(h)}}\otimes \cdots \otimes e_{k_r^{(h)}}\\
&+\,\,\,\,\,\,\,\,\,\,\,\,\,\,\,\,\,\,\,\, \cdots  \,\,\,\,\,\,\,\,\,\,\,\,\,\,\,\,\,\,\,\,  \cdots  \,\,\,\,\,\,\,\,\,\,\,\,\,\,\,\,\,\,\,\, \cdots  \\
&+\sum_{1\leq h \leq p,\,\, i_n<k_1^{(h)}}\a_h \bigg(\sum_{j=1}^n e_{i_j}\bigg)\otimes e_{k_1^{(h)}}\otimes \cdots \otimes e_{k_r^{(h)}}.
\end{align*}
Let us take $\eta:=U_{\theta,J}\xi\in(\cf_J^{\perp})^o$.
It results
\begin{align*}
S_{I,J}^\perp\eta=& \sum_{1\leq h \leq p,\,\, \theta_J(k_1)^{(h)}=i_1}\a_h U_{\theta,J}\big(e_{k_2^{(h)}}\otimes \cdots \otimes e_{k_r^{(h)}}\big) \\
&+\sum_{1\leq h \leq p,\,\, i_1<k_1^{(h)}=\theta_J(k_1^{(h)})<i_2}\a_h e_{i_1}\otimes U_{\theta,J}\big(e_{k_1^{(h)}}\otimes \cdots \otimes e_{k_r^{(h)}}\big) \\
&+\sum_{1\leq h \leq p,\,\, i_1<k_1^{(h)}<\theta_J(k_1^{(h)})=i_2}\a_h \bigg[e_{i_1}\otimes U_{\theta,J}\big(e_{k_1^{(h)}}\otimes \cdots \otimes e_{k_r^{(h)}}\big)\\
&\,\,\,\,\,+U_{\theta,J}\big(e_{k_2^{(h)}}\otimes \cdots \otimes e_{k_r^{(h)}}\big)\bigg] \\
\end{align*}
\begin{align*}
&+\sum_{1\leq h \leq p,\,\, i_1<k_1^{(h)}<i_2<\theta_J(k_1^{(h)})=i_3}\a_h \bigg[(e_{i_1}+e_{i_2})\otimes U_{\theta,J}\big(e_{k_1^{(h)}}\otimes \cdots \otimes e_{k_r^{(h)}}\big)\\
&\,\,\,\,\,+U_{\theta,J}\big(e_{k_2^{(h)}}\otimes \cdots \otimes e_{k_r^{(h)}}\big)\bigg] \\
&+\,\,\,\,\,\,\,\,\,\,\,\,\,\,\,\,\,\,\,\, \cdots  \,\,\,\,\,\,\,\,\,\,\,\,\,\,\,\,\,\,\,\,  \cdots  \,\,\,\,\,\,\,\,\,\,\,\,\,\,\,\,\,\,\,\, \cdots  \\
&+  \sum_{1\leq h \leq p,\,\, i_n<k_1^{(h)}}\a_h \bigg(\sum_{j=1}^n e_{i_j}\bigg)\otimes e_{k_1^{(h)}}\otimes \cdots \otimes e_{k_r^{(h)}}.
\end{align*}
As $U_{\theta,J}$ is unitary, it turns out
$$
\|S_{I,J}\xi\|\leq \|S_{I,J}^\perp \eta\|\leq \|S_{I,J}^\perp\|\|\xi\|,
$$
and the density of $\cf_J^o$ in $\cf_J$ gives $\|S_{I,J}\|\leq \|S_{I,J}^\perp\|$. Since $S_{I,J}^\perp$ is self-adjoint and $\|S_I\|=\sup\{\|S_{I,J}\|,\|S_{I,J}^\perp\|\}$, the thesis follows from Proposition \ref{norm3} and Theorem \ref{norm2}.
\end{proof}

\section{appendix}
\label{app}
In the next lines we briefly present how a different approach with respect to the monotone convolution of Section \ref{sec2} gives us some information about the interlacing structure connecting the atoms of $\m_m$ and those of $\m_1, \ldots, \m_{m-1}$. Here we point out that a similar achievement can be performed just using monotone independence, as shown in \cite{Has}, Theorem 3.1. We decided to put here the following results since they refine the latter case, where interlacing relations are established between $\m_m$ and $\m_{m-1}$, $m\geq 2$.

Let us take
$$
u_{m,n}:=\om_\Om\big((S_m^2)^{n}\big).
$$
As $S_m^2$ are bounded self-adjont operators, the sequence $u_{m,n}$ uniquely determines
the symmetric probability measure, say $\nu_m$, such that
$$
u_{m,n}=\int_\mathbb{R}x^{n}\di \nu_m(x).
$$
Using some results contained in \cite{CFL,CFL2}, one sees the elements of the moment generating functions sequence $(T_m)_m$ are mutually related in the following way
\begin{equation*}
T_{m}\left( t\right) =\frac{1}{1-t-t\sum_{k=2}^{m}T_{k-1}\left(t\right)}.
\end{equation*}
This gives, for each $m$
\begin{equation}
\label{md-04e}
T_{m+1}(t)=\frac{T_{m}(t)}{1-tT^2_{m}(t)}.
\end{equation}
If $M_m(t)$ denotes the moments generating function of $S_m$, from \eqref{md-04e} one achieves
\begin{equation}
\label{md-04f}
M_{m+1}(t)=\frac{M_{m}(t)}{1-t^2M^2_{m}(t)}.
\end{equation}
If
$$
M_1(t)=\frac{Q_1(t)}{P_1(t)}\,,
$$
where $P_{1}(t)=1-t^2$ and $Q_1(t)=1$, and more in general for arbitrary $m$
\begin{equation}
\label{tristar}
M_m(t)=\frac{Q_m(t)}{P_m(t)},
\end{equation}
\eqref{md-04f} yields the following recursive formulas for $m\geq 2$
\begin{equation}
\label{recrel}
Q_{m+1}(t)=\prod_{k=1}^mP_k(t),\,\,\,\, P_{m+1}(t)=P_m^2(t)-t^2Q_m^2(t).
\end{equation}
Thus, properties of the atoms have been reduced to those for roots of the above polynomials. We denote $\mathbb{R}[t]$ the ring of polynomials in the indeterminate $t$ with coefficients in the field $\mathbb{R}$, and $\deg(P)$ is the degree of each $P\in \mathbb{R}[t]$. Furthermore, $\mathcal{Z}_\mathbb{R}(P)$ is the (possibly empty) set of the roots of $P$ in $\mathbb{R}$.
The following lemma is an application of Bolzano's Theorem.
\begin{lemma}
\label{propol1}
Let $P,Q\in \mathbb{R}[t]$ such that $\mathcal{Z}_\mathbb{R}(P)=\{p_1,\ldots, p_m\}$ and $\mathcal{Z}_\mathbb{R}(Q)=\{q_1,\ldots, q_s\}$, for $m,n\geq 1$. Suppose $\mathcal{Z}_\mathbb{R}(P)$ and $\mathcal{Z}_\mathbb{R}(Q)$ are disjoint, and $Q(p_j)>0$ for any $j=1,\ldots, m$. Let $\{r_1,\ldots, r_{m+s}\}$
be the set of the real roots of $P$ and $Q$ taken in an increasing order. If there exists $j$ such that $r_j$ and $r_{j+1}$ are zeros of different polynomials (\emph{i.e} if $r_j\in \mathcal{Z}_\mathbb{R}(P)$, then  $r_{j+1}\in \mathcal{Z}_\mathbb{R}(Q)$, and viceversa), one has $P^2(t)-Q(t)$ possesses at least a real root on $(r_j,r_{j+1})$.
\end{lemma}
Let $(P_n)_n$  and $(Q_n)_n$ be sequences in $\mathbb{R}[t]$ such that
\begin{align}
\begin{split}
\label{p2}
& \deg(P_0)>0,\,\,\,\,\,\, P_0(0)\neq 0,\,\,\,\,\,\, Q_0(t):=1 \\
&Q_{n+1}(t):=\prod_{k=0}^nP_k(t),\,\,\,\,\, P_{n+1}(t):=P_n^2(t)-t^2Q_n^2(t).
\end{split}
\end{align}
\begin{lemma}
\label{propol2}
Let $(P_n)_n$  and $(Q_n)_n$ be sequences satisfying \eqref{p2}. Then, for any $n\geq 0$
\begin{equation}
\label{inters}
\mathcal{Z}_{\mathbb{R}}(P_{n}) \bigcap \mathcal{Z}_{\mathbb{R}}(Q_{n})=\emptyset
\end{equation}
and
$$
\mathcal{Z}_{\mathbb{R}}(Q_{n+1})=\bigcup_{k=0}^n \mathcal{Z}_{\mathbb{R}}(P_{k}), \qquad \mathcal{Z}_{\mathbb{R}}(P_{n})\bigcap \mathcal{Z}_{\mathbb{R}}(P_{m})=\emptyset,\,\,\, \text{if}\,\,\, m\neq n.
$$
\end{lemma}
\begin{proof}
We preliminary notice that if there exists a common root $t_0$ for $P_n$ and $Q_n$ for some $n$, then $t_0$ is not null, as $P_0(0)\neq0$ and \eqref{p2}. Denote now
$$
J:=\{n\in\mathbb{N} \mid P_{n+1}(t_0)=Q_{n+1}(t_0),\,\,\, \text{for some}\,\,\, t_0\in\mathbb{C}\}.
$$
If we prove $J$ is empty, \eqref{inters} follows. In fact, suppose $J\neq \emptyset$ and take $m$ as its minimum. Since $Q_{m+1}(t_0)=0$ for some $t_0$, \eqref{p2} gives that either $P_m(t_0)=0$ or $Q_m(t_0)=0$. The assumption $P_{m+1}(t_0)=0$, together with \eqref{p2} yields that both $P_m$ and $Q_m$ vanish in $t_0$, since $t_0\neq 0$. This contradicts the minimum assumption on $m$.

\noindent Finally, one easily obtains $\mathcal{Z}_{\mathbb{R}}(Q_{n+1})=\bigcup_{k=0}^n \mathcal{Z}_{\mathbb{R}}(P_{k})$. If in addition $n\neq m$, say $n>m$, the last part of the statement follows from \eqref{inters}, as $\mathcal{Z}_{\mathbb{R}}(P_{m})\subseteq \mathcal{Z}_{\mathbb{R}}(Q_{n})$.
\end{proof}
As a consequence, if $P_n(t)$ and $Q_n(t)$ are as in \eqref{recrel}, with $P_{1}(t)=1-t^2$, $Q_1(t)=1$, one has
\begin{align}
\begin{split}
\label{disj}
&\mathcal{Z}_{\mathbb{R}}(P_{n}) \bigcap \mathcal{Z}_{\mathbb{R}}(Q_{n})=\emptyset, \\
&\mathcal{Z}_\mathbb{R}(Q_{n+1})=\bigcup_{k=1}^n \mathcal{Z}_\mathbb{R}(P_{k}), \\
&\mathcal{Z}_\mathbb{R}(P_n)\bigcap\mathcal{Z}_\mathbb{R}(P_m)=\emptyset,\,\,\, \text{if}\,\,\, m\neq n. \\
\end{split}
\end{align}
Finally, we show the interlacing structure connecting the atoms $\m_n$  and those of all $\m_i$, $i\leq n-1$. It is achieved by means of the roots of $P_n(t)$ and $Q_n(t)$. As these ones are both even functions, we reduce the matter to the positive half-plane.
\begin{proposition}
\label{lem2}
Let $P_n(t)$ and $Q_n(t)$ be as in \eqref{recrel}, with $P_{1}(t)=1-t^2$, $Q_1(t)=1$. Then, for any $n\geq 1$ the roots of $P_{n}$ and $Q_{n}$ are real and simple. Moreover, if $p_1^{(n+1)}<\cdots <p_{2^n}^{(n+1)}$ and $q_1^{(n+1)}<\cdots <q_{2^n-1}^{(n+1)}$ are the positive zeros of $P_{n+1}$ and $Q_{n+1}$ respectively, one has
\begin{equation*}
p_1^{(n+1)}<q_1^{(n+1)}<p_2^{(n+1)}<q_2^{(n+1)}<\cdots<p_{2^n-1}^{(n+1)}<q_{2^n-1}^{(n+1)}<p_{2^n}^{(n+1)}.
\end{equation*}
\end{proposition}
\begin{proof}
Indeed, when $n=1$ one finds that the positive zeros of $P_2$ and $Q_2$ are $\frac{\sqrt{5}\pm 1}{2}$ and $1$, respectively.

Now we suppose the statement holds for any $m\leq n$, and consider the case $m=n+1$. As \eqref{recrel} gives
$$
Q_{n+1}(t)=P_n(t)Q_n(t),
$$
any positive root of $Q_{n+1}$ is either a root of $P_n$ or a zero of $Q_n$, and they do not share any zero by the induction assumption. As a consequence, $Q_{n+1}$ has exactly $2^n-1$ positive zeros.
Let $p_1^{(n)}<\cdots <p_{2^{n-1}}^{(n)}$ be the totality of the positive zeros of $P_n$. From \eqref{recrel}
\begin{equation*}
P_{n+1}\big(p_h^{(n)}\big)=-p_h^{(n)}Q^2_n\big(p_h^{(n)}\big) <0,
\end{equation*}
for each $h=1,\ldots,2^{n-1}$, since $Q_n$ and $P_n$ have no common roots, and $Q^2_n(p_h^{(n)})>0$ as follows from \eqref{disj}.
Similarly, for any $h=1,\ldots,2^{n-1}-1$
$$
P_{n+1}\big(q_h^{(n)}\big)=P^2_n\big(q_h^{(n)}\big)>0,
$$
where $q_1^{(n)}<q_2^{(n)}<\cdots <q_{2^{n-1}-1}^{(n)}$ are the positive roots of $Q_n$.
The induction assumption and Lemma \ref{propol1} give us $P_{n+1}$ has a root in each of the intervals $(p_h^{(n)},q_h^{(n)})$ and $(q_h^{(n)},p_{h+1}^{(n)})$,  $h=1,\ldots, 2^{n-1}-1$. The thesis is then achieved as soon as one proves
\begin{equation}
\label{recrel21}
P_{n+1}(t)=t^{2^{n+1}}+ \sum_{k=1}^{2^{n+1}-2}a_k^{(n+1)}t^k+1,
\end{equation}
for all $n\geq 1$. \eqref{recrel21} is indeed satisfied for $n=1$, as $P_2(t)= t^4-3t^2+1$. Further, we assume it holds for any $P_m(t)$ with $m\leq n$. As a consequence,
$P_n^2(t)$ is a monic polynomial with degree $2^{n+1}$ with a constant unital term.

\noindent From \eqref{recrel}, one has
$$
t^2Q_n^2(t)=t^2\big(1-t^2\big)^2\prod_{r=2}^{n-1}\bigg(t^{2^r}+\sum_{k=1}^{2^r-2}a_k^{r}t^k+1\bigg)^2,
$$
and since $\displaystyle\sum_{k=1}^{n-1}2^k=2^n-2$, it follows that $\deg(t^2Q_n^2(t))=2^{n+1}-2$. Finally, \eqref{recrel21} is achieved using again \eqref{recrel}.
\end{proof}

\bigskip

\textbf{Acknowledgements.} The authors kindly acknowledge the Italian INDAM-GNAMPA and Fondi di Ateneo Universit\`a di Bari `Probabilit\`a Quantistica e Applicazioni' for their support. V. Crismale also acknowledges  the FFABR project 2018 of the Italian MIUR.

\end{document}